\documentclass[12pt]{amsart}
\usepackage{amsmath,amsthm,amsfonts,amssymb,mathrsfs}
\usepackage{color}
\date{\today}

\usepackage{hyperref}
\input{xy}
 \xyoption{all}
 \xyoption{arc}

\newtheorem{theorem}{Theorem}

\newtheorem{proposition}{Proposition}
\newtheorem{corollary}{Corollary}

\newtheorem{lemma}{Lemma}
\theoremstyle{definition}

\newtheorem{example}{Example}
\newtheorem{remark}{Remark}

\begin{document}

\title[On feebly compact shift-continuous topologies on the semilattice $\exp_n\lambda$]{On feebly compact shift-continuous topologies on the semilattice $\exp_n\lambda$}
\author{Oleg Gutik and Oleksandra Sobol}
\address{Faculty of Mechanics and Mathematics,
National University of Lviv, Universytetska 1, Lviv, 79000, Ukraine}
\email{o\_\,gutik@franko.lviv.ua, ovgutik@yahoo.com, olesyasobol@mail.ru}



\begin{abstract}
We study feebly compact topologies $\tau$ on the semilattice $\left(\exp_n\lambda,\cap\right)$ such that $\left(\exp_n\lambda,\tau\right)$ is a semitopological semilattice and prove that for any shift-continuous $T_1$-topology $\tau$ on $\exp_n\lambda$ the following conditions are equivalent:
$(i)$~$\tau$ is countably pracompact; $(ii)$ $\tau$ is feebly compact; $(iii)$ $\tau$ is $d$-feebly compact; $(iv)$ $\left(\exp_n\lambda,\tau\right)$ is an $H$-closed space.
\end{abstract}

\maketitle


\begin{center}
  \emph{\textbf{Dedicated to the memory of Professor Vitaly Sushchanskyy}}
\end{center}

\bigskip


We shall follow the terminology of~\cite{Carruth-Hildebrant-Koch-1983-1986, Engelking-1989, Gierz-Hofmann-Keimel-Lawson-Mislove-Scott-2003, Ruppert-1984}. If $X$ is a topological space and $A\subseteq X$, then by $\operatorname{cl}_X(A)$ and $\operatorname{int}_X(A)$ we denote the closure and the interior of $A$ in $X$, respectively. By $\omega$ we denote the first infinite cardinal and by $\mathbb{N}$ the set of positive integers.

A subset $A$ of a topological space $X$ is called \emph{regular open} if $\operatorname{int}_X(\operatorname{cl}_X(A))=A$.

We recall that a topological space $X$ is said to be
\begin{itemize}
  \item \emph{quasiregular} if for any non-empty open set $U\subset X$ there exists a non-empty open set $V\subset U$ such that $\operatorname{cl}_X(V) \subseteq U$;
  \item \emph{semiregular} if $X$ has a base consisting of regular open subsets;
  \item \emph{compact} if each open cover of $X$ has a finite subcover;
  \item \emph{countably compact} if each open countable cover of $X$ has a finite subcover;
  \item \emph{countably compact at a subset} $A\subseteq X$ if every infinite subset $B\subseteq A$  has  an  accumulation  point $x$ in $X$;
  \item \emph{countably pracompact} if there exists a dense subset $A$ in $X$  such that $X$ is countably compact at $A$;
  \item \emph{feebly compact} (or \emph{lightly compact}) if each locally finite open cover of $X$ is finite~\cite{Bagley-Connell-McKnight-Jr-1958};
  \item $d$-\emph{feebly compact} (or \emph{\textsf{DFCC}}) if every discrete family of open subsets in $X$ is finite (see \cite{Matveev-1998});
  \item \emph{pseudocompact} if $X$ is Tychonoff and each continuous real-valued function on $X$ is bounded.
\end{itemize}
According to Theorem~3.10.22 of \cite{Engelking-1989}, a Tychonoff topological space $X$ is feebly compact if and only if $X$ is pseudocompact. Also, a Hausdorff topological space $X$ is feebly compact if and only if every locally finite family of non-empty open subsets of $X$ is finite \cite{Bagley-Connell-McKnight-Jr-1958}.  Every compact space and every sequentially compact space are countably compact, every countably compact space is countably pracompact, and every countably pracompact space is feebly compact (see \cite{Arkhangelskii-1992}), and every $H$-closed space is feebly compact too (see \cite{Gutik-Ravsky-2015a}). Also, it is obvious that every feebly compact space is $d$-feebly compact.

A \emph{semilattice} is a commutative semigroup of idempotents. On  a semilattice $S$ there exists a natural partial order: $e\leqslant f$
\emph{if and only if} $ef=fe=e$. For any element $e$ of a semilattice $S$ we put
\begin{equation*}
  {\uparrow}e=\left\{f\in S\colon e\leqslant f\right\}.
\end{equation*}

A {\it topological} ({\it semitopological}) {\it semilattice} is a topological space together with a continuous (separately continuous) semilattice operation. If $S$ is a~semilattice and $\tau$ is a topology on $S$ such that $(S,\tau)$ is a topological semilattice, then we shall call $\tau$ a \emph{semilattice} \emph{topology} on $S$, and if $\tau$ is a topology on $S$ such that $(S,\tau)$ is a semitopological semilattice, then we shall call $\tau$ a \emph{shift-continuous} \emph{topology} on~$S$.

For an arbitrary positive integer $n$ and an arbitrary non-zero cardinal $\lambda$ we put
\begin{equation*}
  \exp_n\lambda=\left\{A\subseteq \lambda\colon |A|\leqslant n\right\}.
\end{equation*}

It is obvious that for any positive integer $n$ and any non-zero cardinal $\lambda$ the set $\exp_n\lambda$ with the binary operation $\cap$ is a semilattice. Later in this paper by $\exp_n\lambda$ we shall denote the semilattice $\left(\exp_n\lambda,\cap\right)$.

This paper is a continuation of \cite{Gutik-Sobol-2016} where we study feebly compact topologies $\tau$ on the semilattice $\exp_n\lambda$ such that $\left(\exp_n\lambda,\tau\right)$ is a semitopological semilattice. Therein, all compact semilattice $T_1$-topologies on $\exp_n\lambda$ were described. In \cite{Gutik-Sobol-2016} it was proved that for an arbitrary positive integer $n$ and an arbitrary infinite cardinal $\lambda$ every $T_1$-semitopological countably compact semilattice $\left(\exp_n\lambda,\tau\right)$ is a compact topological semilattice. Also, there  we construct a countably pracompact $H$-closed quasiregular non-semiregular topology $\tau_{\operatorname{\textsf{fc}}}^2$ such that $\left(\exp_2\lambda,\tau_{\operatorname{\textsf{fc}}}^2\right)$ is a semitopological semilattice with the discontinuous semilattice operation and show that for an arbitrary positive integer $n$ and an arbitrary infinite cardinal $\lambda$ a semiregular feebly compact semitopological semilattice $\exp_n\lambda$ is a compact topological semilattice.

In this paper we show that for any shift-continuous $T_1$-topology $\tau$ on $\exp_n\lambda$ the following conditions are equivalent:
$(i)$ $\tau$ is countably pracompact; $(ii)$ $\tau$ is feebly compact; $(iii)$ $\tau$ is $d$-feebly compact; $(iv)$~$\left(\exp_n\lambda,\tau\right)$ is an $H$-closed space.


The proof of the following lemma is similar to Lemma~4.5 of~\cite{Bardyla-Gutik-2016} or Proposition~1 from~\cite{Arkhangelskii-1985}.

\begin{lemma}\label{lemma-1}
Every Hausdorff $d$-feebly compact topological space with a dense discrete subspace is countably pracompact.
\end{lemma}

We observe that by Proposition~1 from \cite{Gutik-Sobol-2016} for an arbitrary positive integer $n$ and an arbitrary infinite cardinal $\lambda$ every  shift-continuous $T_1$-topology $\tau$ on $\exp_n\lambda$ is functionally Hausdorff and quasiregular, and hence it is Hausdorff.

\begin{proposition}\label{proposition-2}
Let $n$ be an arbitrary positive integer and $\lambda$ be an arbitrary infinite cardinal. Then for every
$d$-feebly compact shift-continuous $T_1$-topology $\tau$ on $\exp_n\lambda$ the subset
$\exp_n\lambda\setminus\exp_{n-1}\lambda$ is dense in $\left(\exp_n\lambda,\tau\right)$.
\end{proposition}

\begin{proof}
Suppose to the contrary that there exists a $d$-feebly compact
shift-continuous $T_1$-topology $\tau$ on $\exp_n\lambda$ such that
$\exp_n\lambda\setminus\exp_{n-1}\lambda$ is not dense in
$\left(\exp_n\lambda,\tau\right)$. Then there exists a point
$x\in\exp_{n-1}\lambda$ of the space
$\left(\exp_n\lambda,\tau\right)$ such that
$x\notin\operatorname{cl}_{\exp_n\lambda}(\exp_n\lambda\setminus\exp_{n-1}\lambda)$.
This implies that there exists an open neighbourhood $U(x)$ of $x$
in $\left(\exp_n\lambda,\tau\right)$ such that $U(x)\cap
\left(\exp_n\lambda\setminus\exp_{n-1}\lambda\right)=\varnothing$.
The definition of the semilattice $\exp_n\lambda$ implies that every
maximal chain in $\exp_n\lambda$ is finite and hence there exists a
point $y\in U(x)$ such that ${\uparrow}y\cap U(x)=\{y\}$. By
Proposition~1$(iii)$ from \cite{Gutik-Sobol-2016}, ${\uparrow}y$ is
an open-and-closed subset of $\left(\exp_n\lambda,\tau\right)$ and
hence ${\uparrow}y$ is a $d$-feebly compact subspace of
$\left(\exp_n\lambda,\tau\right)$.

It is obvious that the subsemilattice ${\uparrow}y$ of
$\exp_n\lambda$ is algebraically isomorphic to the semilattice
$\exp_k\lambda$ for some positive integer $k\leqslant n$. This and
above arguments imply that without loss of generality we may assume
that $y$ is the isolated zero of the $d$-feebly compact
semitopological semilattice $\left(\exp_n\lambda,\tau\right)$.

Hence we assume that $\tau$ is a $d$-feebly compact shift-continuous
topology on $\exp_n\lambda$ such that the zero $0$ of
$\exp_n\lambda$ is an isolated point of
$\left(\exp_n\lambda,\tau\right)$. Next we fix an arbitrary infinite
sequence $\left\{x_i\right\}_{i\in\mathbb{N}}$ of distinct elements
of cardinal $\lambda$. For every positive integer $j$ we put
\begin{equation*}
  a_j=\left\{x_{n(j-1)+1},x_{n(j-1)+2},\ldots,x_{nj}\right\}.
\end{equation*}
Then $a_j\in\exp_n\lambda$ and moreover $a_j$ is a greatest element of the semilattice $\exp_n\lambda$ for each positive integer $j$. Also, the definition of  the semilattice $\exp_n\lambda$ implies that for every non-zero element $a$ of $\exp_n\lambda$ there exists at most one element $a_j$ such that $a_j\in{\uparrow}a$. Then for every positive integer $j$ by Proposition~1$(iii)$ of \cite{Gutik-Sobol-2016}, $a_j$ is an isolated point of $\left(\exp_n\lambda,\tau\right)$, and hence the above arguments imply that $\left\{a_1,a_2,\ldots,a_j,\ldots\right\}$ is an infinite discrete family of open subset in the space $\left(\exp_n\lambda,\tau\right)$. This contradicts the $d$-feeble compactness of the semitopological semilattice $\left(\exp_n\lambda,\tau\right)$. The obtained contradiction implies the statement of our proposition.  
\end{proof}

The following example show that the converse statement to Proposition~\ref{proposition-2} is not true in the case of topological semilattices.

\begin{example}\label{example-3}
Fix an arbitrary cardinal $\lambda$ and an infinite subset $A$ in
$\lambda$ such that $\left|\lambda\setminus
A\right|\geqslant\omega$. By $\pi\colon\lambda\to\exp_1\lambda\colon
a\mapsto\left\{a\right\}$ we denote the natural embedding of
$\lambda$ into $\exp_1\lambda$. On $\exp_1\lambda$ we define a
topology $\tau_{\operatorname{\textsf{dm}}}$ in the following way:
\begin{itemize}
  \item[$(i)$] all non-zero elements of the semilattice $\exp_1\lambda$ are isolated points in
  $\left(\exp_1\lambda,\tau_{\operatorname{\textsf{dm}}}\right)$; \; and
  \item[$(ii)$] the family $\mathscr{B}_{\operatorname{\textsf{dm}}}=\left\{U_B=\{0\}\cup\pi(B)\colon B\subseteq A \hbox{~and~} A\setminus B \hbox{~is finite}\right\}$ is the base of the topology $\tau_{\operatorname{\textsf{dm}}}$ at zero $0$ of $\exp_1\lambda$.
\end{itemize}
Simple verifications show that $\tau_{\operatorname{\textsf{dm}}}$ is a Hausdorff locally compact semilattice topology on $\exp_1\lambda$ which is not compact and hence by Corollary~8 of \cite{Gutik-Sobol-2016} it is not feebly compact.
\end{example}

\begin{remark}\label{remark-4}
We observe that in the case when $\lambda=\omega$ by Proposition~13 of \cite{Gutik-Sobol-2016} the topological space $\left(\exp_1\lambda,\tau_{\operatorname{\textsf{dm}}}\right)$ is collectionwise normal and it has a countable base, and hence $\left(\exp_1\lambda,\tau_{\operatorname{\textsf{dm}}}\right)$ is metrizable by the Urysohn Metrization Theorem \cite{Urysohn-1925}. Moreover, if $|B|=\omega$ then the space $\left(\exp_1\lambda,\tau_{\operatorname{\textsf{dm}}}\right)$ is metrizable for any infinite cardinal $\lambda$, as a topological sum of the metrizable space $\left(\exp_1\omega,\tau_{\operatorname{\textsf{dm}}}\right)$ and the discrete space of cardinality~$\lambda$.
\end{remark}

\begin{remark}\label{remark-5}
If $n$ is an arbitrary positive integer $\geqslant 3$, $\lambda$ is
any infinite cardinal and $\tau_{\operatorname{\textsf{c}}}^n$ is
the unique compact semilattice topology on the semilattice
$\exp_n\lambda$ defined in Example~4 of \cite{Gutik-Sobol-2016},
then we construct more stronger topology
$\tau_{\operatorname{\textsf{dm}}}^n$ on $\exp_n\lambda$ them
$\tau_{\operatorname{\textsf{c}}}^n$ in the following way. Fix an
arbitrary element $x\in \exp_n\lambda$ such that $|x|=n-1$. It is
easy to see that the subsemilattice ${\uparrow}x$ of $\exp_n\lambda$
is isomorphic to $\exp_1\lambda$, and by
$h\colon\exp_1\lambda\to{\uparrow}x$ we denote this isomorphism.

Fix an arbitrary subset $A$ in $\lambda$ such that
$\left|\lambda\setminus A\right|\geqslant\omega$. For every zero
element $y\in\exp_n\lambda\setminus{\uparrow}x$ we assume that the base
$\mathscr{B}_{\operatorname{\textsf{dm}}}^n(y)$ of the topology
$\tau_{\operatorname{\textsf{dm}}}^n$ at the point $y$ coincides
with the base of the topology $\tau_{\operatorname{\textsf{c}}}^n$
at $y$, and assume that ${\uparrow}x$ is an open-and-closed subset and the
topology on ${\uparrow}x$ is generated by the map
$h\colon\left(\exp_2\lambda,\tau_{\operatorname{\textsf{fc}}}^2\right)\to{\uparrow}x$.
We observe that
$\left(\exp_n\lambda,\tau_{\operatorname{\textsf{dm}}}^n\right)$ is
a Hausdorff locally compact topological space, because it is the
topological sum of a Hausdorff locally compact space ${\uparrow}x$
(which is homeomorphic to the Hausdorff locally compact space
$\left(\exp_1\lambda,\tau_{\operatorname{\textsf{dm}}}\right)$ from
Example~\ref{example-3}) and an open-and-closed subspace
$\exp_n\lambda\setminus{\uparrow}x$ of
$\left(\exp_n\lambda,\tau_{\operatorname{\textsf{c}}}^n\right)$. It
is obvious that the set $\exp_n\lambda\setminus\exp_{n-1}\lambda$ is
dense in
$\left(\exp_n\lambda,\tau_{\operatorname{\textsf{dm}}}^n\right)$.
Also, since ${\uparrow}x$ is an open-and-closed subsemilattice with
zero $x$ of
$\left(\exp_n\lambda,\tau_{\operatorname{\textsf{dm}}}^n\right)$,
the continuity of the semilattice operations in
$\left(\exp_n\lambda,\tau_{\operatorname{\textsf{dm}}}^n\right)$ and
$\left(\exp_n\lambda,\tau_{\operatorname{\textsf{c}}}^n\right)$ and
the property that the topology $\tau_{\operatorname{\textsf{dm}}}^n$
is more stronger them $\tau_{\operatorname{\textsf{c}}}^n$, imply
that
$\left(\exp_n\lambda,\tau_{\operatorname{\textsf{dm}}}^n\right)$ is
a topological semilattice.  Moreover, the space
$\left(\exp_n\lambda,\tau_{\operatorname{\textsf{dm}}}^n\right)$ is
not $d$-feebly compact, because it contains an open-and-closed
non-$d$-feebly compact subspace ${\uparrow}x$.
\end{remark}

Arguments presented in the proof of Proposition~\ref{proposition-2} and Proposition~1$(iii)$ of \cite{Gutik-Sobol-2016} imply the following corollary.

\begin{corollary}\label{corollary-6}
Let $n$ be an arbitrary positive integer and $\lambda$ be an arbitrary infinite cardinal. Then for every $d$-feebly compact shift-continuous $T_1$-topology $\tau$ on $\exp_n\lambda$ a point $x$ is isolated in $\left(\exp_n\lambda,\tau\right)$ if and only if $x\in\exp_n\lambda\setminus\exp_{n-1}\lambda$.
\end{corollary}

\begin{remark}\label{remark-7}
We observe that the example presented in Remark~\ref{remark-5} implies there exists a locally compact non-$d$-feebly compact semitopological semilattice $\left(\exp_n\lambda,\tau_{\operatorname{\textsf{dm}}}^n\right)$ with the following property: \emph{a point $x$ is isolated in $\left(\exp_n\lambda,\tau_{\operatorname{\textsf{dm}}}^n\right)$ if and only if $x\in\exp_n\lambda\setminus\exp_{n-1}\lambda$}.
\end{remark}

The following proposition gives an amazing property of the system of neighbourhoodd of zero in a $T_1$-feebly compact semitopological semilattice $\exp_n\lambda$.

\begin{proposition}\label{proposition-8}
Let $n$ be an arbitrary positive integer, $\lambda$ be an arbitrary infinite cardinal and $\tau$ be a shift-continuous feebly compact $T_1$-topology on the semilattice $\exp_n\lambda$. Then for every open neighbourhood $U(0)$ of zero $0$ in $\left(\exp_n\lambda,\tau\right)$ there exist finitely many $x_1,\ldots,x_m\in\lambda$ such that
\begin{equation*}
\exp_n\lambda\setminus\operatorname{cl}_{\exp_n\lambda}(U(0))\subseteq{\uparrow}x_1\cup\cdots\cup{\uparrow}x_m.
\end{equation*}
\end{proposition}

\begin{proof}
Suppose to the contrary that there exists an open neighbourhood $U(0)$ of zero in a Hausdorff feebly compact semitopological semilattice $\left(\exp_n\lambda,\tau\right)$ such that
\begin{equation*}
\exp_n\lambda\setminus\operatorname{cl}_{\exp_n\lambda}(U(0))\not\subseteq{\uparrow}x_1\cup\cdots\cup{\uparrow}x_m
\end{equation*}
for any finitely many $x_1,\ldots,x_m\in\lambda$.

We fix an arbitrary $y_1\in\lambda$ such that $\left(\exp_n\lambda\setminus\operatorname{cl}_{\exp_n\lambda}(U(0))\right)\cap{\uparrow}y_1\neq\varnothing$. By Proposition~1$(iii)$ of \cite{Gutik-Sobol-2016} the set ${\uparrow}y_1$ is open in $\left(\exp_n\lambda,\tau\right)$ and hence the set $\left(\exp_n\lambda\setminus\operatorname{cl}_{\exp_n\lambda}(U(0))\right)\cap{\uparrow}y_1$ is open in $\left(\exp_n\lambda,\tau\right)$ too. Then by Proposition~\ref{proposition-2} there exists an isolated point $m_1\in\exp_n\lambda\setminus\exp_{n-1}\lambda$ in $\left(\exp_n\lambda,\tau\right)$ such that $m_1\in\left(\exp_n\lambda\setminus\operatorname{cl}_{\exp_n\lambda}(U(0))\right)\cap{\uparrow}y_1$. Now, by the assumption there exists $y_2\in\lambda$~such~that
\begin{equation*}
\left(\exp_n\lambda\setminus\operatorname{cl}_{\exp_n\lambda}(U(0))\right)\cap\left({\uparrow}y_2 \setminus{\uparrow}y_1\right)\neq\varnothing.
\end{equation*}
Again, since by Proposition~1$(iii)$ of \cite{Gutik-Sobol-2016} both sets ${\uparrow}y_1$ and ${\uparrow}y_2$ are open-and-closed in $\left(\exp_n\lambda,\tau\right)$, Proposition~\ref{proposition-2} implies that there exists an isolated point $m_2\in\exp_n\lambda\setminus\exp_{n-1}\lambda$ in $\left(\exp_n\lambda,\tau\right)$ such that
\begin{equation*}
m_2\in\left(\exp_n\lambda\setminus\operatorname{cl}_{\exp_n\lambda}(U(0))\right)\cap\left({\uparrow}y_2 \setminus{\uparrow}y_1\right).
\end{equation*}

Hence by induction we can construct a sequence $\left\{y_i\colon i=1,2,3,\ldots\right\}$ of distinct points of $\lambda$ and a sequence of isolated points $\left\{m_i\colon i=1,2,3,\ldots\right\}\subset\exp_n\lambda\setminus\exp_{n-1}\lambda$ in $\left(\exp_n\lambda,\tau\right)$ such that for any positive integer $k$ the following conditions hold:
\begin{itemize}
  \item[$(i)$] $\left(\exp_n\lambda\setminus\operatorname{cl}_{\exp_n\lambda}(U(0))\right)\cap\left({\uparrow}y_k \setminus\left({\uparrow}y_1\cup\cdots\cup{\uparrow}y_{k-1}\right)\right)\neq\varnothing$; \; and
  \item[$(ii)$] $m_k\in\left(\exp_n\lambda\setminus\operatorname{cl}_{\exp_n\lambda}(U(0))\right)\cap\left({\uparrow}y_k \setminus\left({\uparrow}y_1\cup\cdots\cup{\uparrow}y_{k-1}\right)\right)$.
\end{itemize}
Then similar arguments as in the proof of Proposition~\ref{proposition-2} imply that the following family
\begin{equation*}
\left\{\left\{m_i\right\}\colon i=1,2,3,\ldots\right\}
\end{equation*}
is infinite and locally finite, which contradicts the feeble compactness of $\left(\exp_n\lambda,\tau\right)$. The obtained contradiction implies the statement of the proposition.  
\end{proof}

Proposition~1$(iii)$ of \cite{Gutik-Sobol-2016} implies that for any element $x\in\exp_n\lambda$ the set ${\uparrow}x$ is open-and-closed in a $T_1$-semitopological semilattice $\left(\exp_n\lambda,\tau\right)$ and hence by Theorem~14 from \cite{Bagley-Connell-McKnight-Jr-1958} we have that for any $x\in\exp_n\lambda$ the space ${\uparrow}x$ is feebly compact in a feebly compact $T_1$-semitopological semilattice $\left(\exp_n\lambda,\tau\right)$. Hence Proposition~\ref{proposition-8} implies the following proposition.

\begin{proposition}\label{proposition-9}
Let $n$ be an arbitrary positive integer, $\lambda$ be an arbitrary infinite cardinal and $\tau$ be a shift-continuous feebly compact $T_1$-topology on the semilattice $\exp_n\lambda$. Then for any point $x\in\exp_n\lambda$ and any open neighbourhood $U(x)$ of $x$ in $\left(\exp_n\lambda,\tau\right)$ there exist finitely many $x_1,\ldots,x_m\in{\uparrow}x\setminus\{x\}$ such that
\begin{equation*}
{\uparrow}x\setminus\operatorname{cl}_{\exp_n\lambda}(U(x))\subseteq{\uparrow}x_1\cup\cdots\cup{\uparrow}x_m.
\end{equation*}
\end{proposition}

The main results of this paper is the following theorem.

\begin{theorem}\label{theorem-9}
Let $n$ be an arbitrary positive integer and $\lambda$ be an arbitrary infinite cardinal. Then for any shift-continuous $T_1$-topology $\tau$ on $\exp_n\lambda$ the following conditions are equivalent:
\begin{itemize}
  \item[$(i)$] $\tau$ is countably pracompact;
  \item[$(ii)$] $\tau$ is feebly compact;
  \item[$(iii)$] $\tau$ is $d$-feebly compact;
  \item[$(iv)$] the space $\left(\exp_n\lambda,\tau\right)$ is $H$-closed.
\end{itemize}
\end{theorem}

\begin{proof}
Implications $(i)\Rightarrow(ii)$ and $(ii)\Rightarrow(iii)$ are trivial and  implication $(iii)\Rightarrow(i)$ follows from Proposition~1 of \cite{Gutik-Sobol-2016}, Lemma~\ref{lemma-1} and Proposition~\ref{proposition-2}.

Implication $(iv)\Rightarrow(ii)$ follows from Proposition~4 of \cite{Gutik-Ravsky-2015a}.

$(ii)\Rightarrow(iv)$ We shall prove this implication by induction.

By Corollary~2 from \cite{Gutik-Sobol-2016} every feebly compact $T_1$-topology $\tau$ on the semilattice $\exp_1\lambda$ such that $\left(\exp_1\lambda,\tau\right)$ is a semitopological semilattice, is compact, and hence $\left(\exp_1\lambda,\tau\right)$ is an $H$-closed topological space.

Next we shall show that if our statements holds for all positive integers $j<k\leqslant n$ then it holds for $j=k$. Suppose that a feebly compact $T_1$-semitopological semilattice $\left(\exp_k\lambda,\tau\right)$ is a subspace of Hausdorff topological space $X$. Fix an arbitrary point $x\in X$ and an arbitrary open neighbourhood $V(x)$ of $x$ in $X$. Since $X$ is Hausdorff, there exist disjoint open neighbourhoods $U(x)\subseteq V(x)$ and $U(0)$ of $x$ and zero $0$ of the semilattice $\exp_k\lambda$ in $X$, respectively. Then $\operatorname{cl}_X(U(0))\cap U(x)=\varnothing$ and hence by Proposition~\ref{proposition-8} there exists finitely many $x_1,\ldots,x_m\in\lambda$ such that
\begin{equation*}
\exp_k\lambda\cap U(x)\subseteq{\uparrow}x_1\cup\cdots\cup{\uparrow}x_m.
\end{equation*}
But for any $x\in\lambda$ the subsemilattice ${\uparrow}x$ of $\exp_k\lambda$ is algebraically isomorphic to the semilattice $\exp_{k-1}\lambda$. Then by Proposition~1$(iii)$ of \cite{Gutik-Sobol-2016} and Theorem~14 from \cite{Bagley-Connell-McKnight-Jr-1958}, ${\uparrow}x$ is a feebly compact $T_1$-semitopological semilattice, and the assumption of our induction implies that ${\uparrow}x_1,\cdots,{\uparrow}x_m$ are closed subsets of $X$. This implies that
\begin{equation*}
W(x)=U(x)\setminus\left({\uparrow}x_1\cup\cdots\cup{\uparrow}x_m\right)
\end{equation*}
is an open neighbourhood of $x$ in $X$ such that $W(x)\cap \exp_k\lambda=\varnothing$. Thus, $\left(\exp_k\lambda,\tau\right)$ is an $H$-closed space. This completes the proof of the requested implication.  
\end{proof}

The following theorem gives a sufficient condition when a $d$-feebly compact space is feebly compact.

\begin{theorem}\label{theorem-10}
Every quasiregular $d$-feebly compact space is feebly compact.
\end{theorem}

\begin{proof}
Suppose to the contrary that there exists a quasiregular $d$-feebly compact space $X$ which is not feebly compact. Then there exists an infinite locally finite family $\mathscr{U}_0$ of non-empty open subsets of~$X$.

By induction we shall construct an infinite discrete family of non-empty open subsets of $X$.

Fix an arbitrary $U_1\in\mathscr{U}_0$ and an arbitrary point $x_1\in U_1$. Since the family $\mathscr{U}_0$ is locally finite there exists an open neighbourhood $U(x_1)\subseteq U_1$ of the point $x_1$ in $X$ such that $U(x_1)$ intersects finitely many elements of $\mathscr{U}_0$. Also, the quasiregularity of $X$ implies that there exists a non-empty open subset $V_1\subseteq U(x_1)$ such that $\operatorname{cl}_X(V_1)\subseteq U(x_1)$. Put
\begin{equation*}
\mathscr{U}_1=\left\{U\in\mathscr{U}_0\colon U(x_1)\cap U=\varnothing\right\}.
\end{equation*}
Since the family $\mathscr{U}_0$ is locally finite and infinite, so is $\mathscr{U}_1$. Fix an arbitrary $U_2\in\mathscr{U}_1$ and an arbitrary point $x_2\in U_2$. Since the family $\mathscr{U}_1$ is locally finite, there exists an open neighbourhood $U(x_2)\subseteq U_2$ of the point $x_2$ in $X$ such that $U(x_2)$ intersects finitely many elements of $\mathscr{U}_1$. Since $X$ is quasiregular, there exists a non-empty open subset $V_2\subseteq U(x_2)$ such that $\operatorname{cl}_X(V_2)\subseteq U(x_2)$. Our construction implies that the closed sets $\operatorname{cl}_X(V_1)$ and $\operatorname{cl}_X(V_2)$ are disjoint and hence so are  $V_1$ and $V_2$. Next we put
\begin{equation*}
\mathscr{U}_2=\left\{U\in\mathscr{U}_1\colon U(x_2)\cap U=\varnothing\right\}.
\end{equation*}

Also, we observe that it is obvious that $U(x_1)\cap U=\varnothing$ for each $U\in\mathscr{U}_1$.

Suppose for some positive integer $k>1$ we construct:
\begin{itemize}
  \item[$(a)$] a sequence of infinite locally finite subfamilies $\mathscr{U}_1,\ldots, \mathscr{U}_{k-1}$
  in  $\mathscr{U}_0$ of non-empty open subsets in the space $X$;
  \item[$(b)$] a sequence of open subsets $U_1,\ldots, U_k$ in $X$;
  \item[$(c)$] a sequence of points $x_1,\ldots, x_k$ in $X$ and a sequence of their corresponding open neighbourhoods $U(x_1),\ldots,U(x_k)$ in $X$;
  \item[$(d)$] a sequence of disjoint non-empty subsets $V_1,\ldots, V_k$ in $X$
\end{itemize}
such that the following conditions hold:
\begin{itemize}
  \item[$(i)$] $\mathscr{U}_i$ is a proper subfamily of $\mathscr{U}_{i-1}$;
  \item[$(ii)$] $U_i\in \mathscr{U}_{i-1}$ and $U_i\cap U=\varnothing$ for each $U\in \mathscr{U}_j$ with $i\leqslant j\leqslant k$;
  \item[$(iii)$] $x_i\in U_i$ and $U(x_{i})\subseteq U_i$;
  \item[$(iv)$] $V_{i}$ is an open subset of $U_i$ with $\operatorname{cl}_X(V_i)\subseteq U(x_i)$,
\end{itemize}
for all $i=1,\ldots,k$, and
\begin{itemize}
  \item[$(v)$] $\operatorname{cl}_X(V_1),\ldots,\operatorname{cl}_X(V_k)$ are disjoint.
\end{itemize}

Next we put
\begin{equation*}
\mathscr{U}_k=\left\{U\in\mathscr{U}_{k-1}\colon U(x_1)\cap U=\ldots=U(x_k)\cap U=\varnothing\right\}.
\end{equation*}
Since the family $\mathscr{U}_{k-1}$ is infinite and locally finite, there exists a subfamily $\mathscr{U}_k$ in $\mathscr{U}_{k-1}$ which is infinite and locally finite. Fix an arbitrary $U_{k+1}\in\mathscr{U}_k$ and an arbitrary point $x_{k+1}\in U_{k+1}$. Since the family $\mathscr{U}_{k}$ is locally finite, there exists an open neighbourhood $U(x_{k+1})\subseteq U_{k+1}$ of the point $x_{k+1}$ in $X$ such that $U(x_{k+1})$ intersects finitely many elements of $\mathscr{U}_{k}$. Since the space $X$ is quasiregular, there exists a non-empty open subset $V_{k+1}\subseteq U(x_{k+1})$ such that $\operatorname{cl}_X(V_{k+1})\subseteq U(x_{k+1})$. Simple verifications show that the conditions $(i)-(iv)$ hold in the case of the positive integer $k+1$.

Hence by induction we construct the following two infinite countable families of open non-empty subsets of $X$:
\begin{equation*}
\mathscr{U}=\left\{U_i\colon i=1,2,3,\ldots\right\} \qquad \hbox{and} \qquad \mathscr{V}=\left\{V_i\colon i=1,2,3,\ldots\right\}
\end{equation*}
such that $\operatorname{cl}_X(V_i)\subseteq U_i$ for each positive integer $i$. Since $\mathscr{U}$ is a subfamily of $\mathscr{U}_0$ and $\mathscr{U}_0$ is locally finite in $X$, $\mathscr{U}$ is locally finite  in $X$ as well. Also, above arguments imply that $\mathscr{V}$ and
\begin{equation*}
  \overline{\mathscr{V}}=\left\{\operatorname{cl}_X(V_i)\colon i=1,2,3,\ldots\right\}
\end{equation*}
are locally finite families in $X$ too.

Next we shall show that the family $\mathscr{V}$ is discrete in $X$.
Indeed, since the family $\overline{\mathscr{V}}$ is locally finite
in $X$, by Theorem~1.1.11 of~\cite{Engelking-1989} the union
$\bigcup\overline{\mathscr{V}}$ is a closed subset of $X$, and hence
any point $x\in X\setminus\bigcup\overline{\mathscr{V}}$ has an open
neighbourhood $O(x)=X\setminus\bigcup\overline{\mathscr{V}}$ which
does not intersect the elements of the family $\mathscr{V}$. If
$x\in \operatorname{cl}_X(V_i)$ for some positive integer $i$, then
our construction implies that $U(x_i)$ is an open neighbourhood of
$x$ which intersects only the set $V_i\in\mathscr{V}$. Hence $X$ has
an infinite discrete family $\mathscr{V}$ of non-empty open subsets
in $X$, which contradicts the assumption that the space $X$ is
$d$-feebly compact. The obtained contradiction implies the statement
of the theorem.  
\end{proof}

We finish this note by some simple remarks about dense embedding of an infinite semigroup of matrix units and a polycyclic monoid into $d$-feebly compact topological semigroups which follow from the results of the paper \cite{Bardyla-Gutik-2016}.

Let $\lambda$ be a non-zero cardinal. On the set
 $
 B_{\lambda}=(\lambda\times\lambda)\cup\{ 0\}
 $,
where $0\notin\lambda\times\lambda$, we define the semigroup
operation ``$\, \cdot\, $'' as follows
\begin{equation*}
(a, b)\cdot(c, d)=
\left\{
  \begin{array}{cl}
    (a, d), & \hbox{ if~ } b=c;\\
    0, & \hbox{ if~ } b\neq c,
  \end{array}
\right.
\end{equation*}
and $(a, b)\cdot 0=0\cdot(a, b)=0\cdot 0=0$ for $a,b,c,d\in\lambda$. The semigroup $B_{\lambda}$ is called the \emph{semigroup of $\lambda{\times}\lambda$-matrix units}~(see \cite{Clifford-Preston-1961-1967}).

The bicyclic monoid ${\mathscr{C}}(p,q)$ is the semigroup with the identity $1$ generated by two elements $p$ and $q$ subjected only to the condition $pq=1$ \cite{Clifford-Preston-1961-1967}. For a non-zero cardinal $\lambda$, the polycyclic monoid $P_\lambda$ on $\lambda$ generators is the semigroup with zero given by the presentation:
\begin{equation*}
    P_\lambda=\left\langle \left\{p_i\right\}_{i\in\lambda}, \left\{p_i^{-1}\right\}_{i\in\lambda}\mid p_i p_i^{-1}=1, p_ip_j^{-1}=0 \hbox{~for~} i\neq j\right\rangle
\end{equation*}
(see \cite{Bardyla-Gutik-2016}). It is obvious that in the case when $\lambda=1$ the semigroup $P_1$ is isomorphic to the bicyclic semigroup with adjoined zero.

By Theorem~4.4 from \cite{Bardyla-Gutik-2016} for every infinite cardinal $\lambda$ the semigroup of $\lambda{\times}\lambda$-matrix units $B_{\lambda}$ does not densely embed into a Hausdorff feebly compact topological semigroup, and by Theorem~4.5 from \cite{Bardyla-Gutik-2016} for arbitrary cardinal $\lambda\geqslant 2$ there exists no  Hausdorff feebly compact topological semigroup which contains the $\lambda$-polycyclic monoid $P_{\lambda}$ as a dense subsemigroup. These theorems and Lemma~\ref{lemma-1} imply the following two corollaries.

\begin{corollary}
For every infinite cardinal $\lambda$ the semigroup of $\lambda{\times}\lambda$-matrix units $B_{\lambda}$ does not densely embed into a Hausdorff $d$-feebly compact topological semigroup.
\end{corollary}

\begin{corollary}
For arbitrary cardinal $\lambda\geqslant 2$ there exists no  Hausdorff $d$-feebly compact topological semigroup which contains the $\lambda$-polycyclic monoid $P_{\lambda}$ as a dense subsemigroup.
\end{corollary}

The proof of the following corollary is similar to Theorem~5.1(5) from \cite{Banakh-Dimitrova-Gutik-2010}.

\begin{corollary}
There exists no  Hausdorff topological semigroup with the $d$-feebly compact square which contains the bicyclic monoid ${\mathscr{C}}(p,q)$ as a dense subsemigroup.
\end{corollary}

\section*{Acknowledgements}

We acknowledge Alex Ravsky and the referee for their comments and
suggestions.

\end{document}